\documentclass[11pt]{amsart}

\usepackage{amsmath,amssymb}
\usepackage{accents}
\usepackage{graphicx}
\usepackage{verbatim}
\usepackage{enumitem}
\usepackage{tikz}
\usepackage{caption}
\usepackage{subcaption}
\usepackage{tabu}
\usepackage{longtable}
\usepackage{array}
\usepackage{hyperref}

\newlength{\dhatheight}

\theoremstyle{definition}

\newtheorem{theorem}{Theorem}[section]
\newtheorem{conjecture}[theorem]{Conjecture}
\newtheorem{proposition}[theorem]{Proposition}
\newtheorem{lemma}[theorem]{Lemma}

\newtheorem{corollary}[theorem]{Corollary}
\newtheorem{definition}[theorem]{Definition}

\newtheorem{example}[theorem]{Example}

\newtheorem{question}[theorem]{Question}

\newcommand{\ep}{\epsilon}

\newcommand{\s}{\sigma}

\newlist{pcases}{enumerate}{1}
\setlist[pcases]{
  label={\em{Case~\arabic*:}}\protect\thiscase.~,
  ref=\arabic*,
  align=left,
  labelsep=0pt,
  leftmargin=0pt,
  labelwidth=0pt,
  parsep=0pt
}
\newcommand{\case}[1][]{%
  \if\relax\detokenize{#1}\relax
    \def\thiscase{}%
  \else
    \def\thiscase{~#1}%
  \fi
  \item
}

\begin{document}
\title[Families of not perfectly straight knots]
{Families of not perfectly straight knots}

\author[N.~Owad]{Nicholas Owad}
\address{Topology and Geometry of Manifolds Unit\\
Okinawa Institute of Science and Technology Graduate University\\
Okinawa, Japan 904-0495}
\email{nicholas.owad@oist.jp}

\begin{abstract}
We present two families of knots which have straight number higher than crossing number.  In the case of the second family, we have computed the straight number explicitly.  We also give a general theorem about alternating knots that states adding an even number of crossings to a twist region will not change whether the knots are perfectly straight or not perfectly straight.
\end{abstract}

\maketitle



\section{Introduction}\label{sec:intro}


Knot diagrams are most commonly drawn with the minimum number of crossings.  This is how they appear in the knot table in Rolfsen \cite{Rolfsen} which is often referred to as the standard knot table.  Other common ways of presenting knots are with braids closures, in bridge position, thin position, \"ubercrossing and petal diagrams and numerous others.  From most of these presentations of diagrams, invariants are created which are interesting in their own respect.  Jablan and Radovi\'c defined the Meander number and OGC number, see \cite{JR}.  The author answered questions of theirs in \cite{me} and defined the invariant, the straight number of a knot.  In these two papers, Jablan and Radovi\'c and the author only succeed in calculating the straight number for the standard table of knots and a few simple families with straight number equal to the crossing number.  Here, we present the first known infinite families of knots with straight number strictly larger than crossing number.

Adams, Shinjo, and Tanka in \cite{AST} have the following result which we make use of here.

\begin{theorem}\label{thm:AST}{\cite[Theorem 1.2]{AST}}
Every knot has a projection that can be decomposed into two sub-arcs such that each sub-arc never crosses itself.
\end{theorem}

From this result, via a planar isotopy, one can produce a diagram with a single straight strand that contains all of the crossings, and we say the diagram is in {\em straight position}. By convention, we will draw this straight arc horizontally.  The number of crossings might need to increase to draw a diagram in straight position, and so we say the minimum number of crossings over all straight diagrams for a knot $K$ is the straight number,  $\mathtt{str}(K)$.

In \cite{me}, we calculated the straight number of all the knots in the standard Rolfsen table \cite{Rolfsen}.  We say that a knot $K$ is perfectly straight if $\mathtt{str}(K) = c(K)$.  Also in \cite{me}, we proved a few basic families of knots are perfectly straight, including torus knots $T_{2,q}$, pretzel knots, and 2-bridge knots with a continued fraction decomposition length less than 6.   But these results relied on finding an arc which meets every crossing before meeting itself a second time for some nice diagrams.  It is a harder problem to prove that a family of knots is not perfectly straight, as noted in \cite{me} by the following question. 

\begin{question}\cite{me}
 Can we find families of knots which are not perfectly straight?
\end{question}

In this paper, we produce two new families of knots which are not perfectly straight.  The second family comes from a more general theorem about twist regions in knots and straight number, which can be used on any alternating knot to create a new family, for either perfectly straight or not perfectly straight.

A spiral knot  $S(n,m, \ep)$, first defined in \cite{spiral}, see Definition \ref{def:spiral},  is a generalization of torus knots by changing the crossings in the standard braid diagram.  Champanerkar, Kofman, and Purcell, \cite{weave}, defined weaving knots, $W(n,m)$, which are alternating spiral knots.

\noindent{\bf{Theorem}~\ref{thm:spiral}.} 
{
Let $n\geq 3, m\geq n+1$ and $\gcd(n,m) = 1$.   Every weaving knot $W(n,m)$, is not perfectly straight, i.e. $\mathtt{str}(W(n,m))>c(W(n,m)).$
}

The next theorem is much more general and lets us create new infinite families under an operation we call increasing the number of full twists, see Definition \ref{def:incfull}.  Loosely, increasing the number of full twists means adding in an even number of crossings to a twist region.

\noindent{\bf{Theorem}~\ref{thm:inctwist}.}
{
Let $K$ be an alternating knot. Given any minimal diagram $D$ of $K$, let $K'$ be the knot obtained by increasing the number of full twists in any twist region of $D$.  Then $K$ is perfectly straight if and only if $K'$ is perfectly straight.
}

By applying this theorem, we can generalize Theorem \ref{thm:spiral}.  

\noindent{\bf{Corollary}~\ref{cor:spiral}.}
{
Let $n\geq 3, m\geq n+1$ and $\gcd(n,m)=1$.  Let  $w_i = ( \s_1^{\ep_{i_1}}\s_2^{\ep_{i_2}}\cdots \s_{n-1}^{\ep_{i_{n-1}}})$ where each $\ep_{i_j}$ is an odd integers.  Then let $w = w_1w_2\cdots w_m$ and let $K$ be the closure of $w$.  If $K$ is alternating, then $K$ is not perfectly straight.  
}

We also use Theorem \ref{thm:inctwist} on the knot $9_{32}$, which has the property $\mathtt{str}(9_{32}) = c(9_{32})+1 = 10$, to create the first known family of not perfectly straight knots with known straight number.

\noindent{\bf{Theorem}~\ref{thm:template}.}
{
Let $t=(t_1t_2,\ldots,t_6)$ be positive integers such that $t_1,t_2,t_5,$ and $t_6$  are odd and $t_3$ and $t_4$ are even and let $s$ be the sum of the $t_i$'s. Let $K_t$ be the alternating knot obtained from the template in Figure \ref{fig:template} with $t_i$ crossings in the corresponding twist region.  Then $\mathtt{str}(K_t) = c(K_t)+1 =  s+2$.
}

In the next section, we give some basic definitions related to straight number and braids.  In Section \ref{sec:braids}, we define and prove the family of weaving knots is not perfectly straight.  And in Section \ref{sec:twists} we  investigate how increasing the number of full twists affects the straight number.  Here we give the second family of not perfectly straight knots, and find their straight number.  We also discuss generalizing this idea to make any number of infinite families of not perfectly straight knots.


\section{Definitions and Background } \label{sec:notation}

We assume the reader is familiar with braids and knot theory.  See Birman and Brendles \cite{BB} and Rolfsen \cite{Rolfsen}, respectively, for more information.  For more information on straight knots and their properties, see \cite{me}.  We give the definitions that are relevant for this paper here.

A link is an ambient isotopy class of $n$ embedded circles in 3-space, i.e. $\bigsqcup^n S^1\hookrightarrow S^3$.  A knot is a link with a single component, $n=1$.  By Theorem \ref{thm:AST}, we know that every knot can be drawn with two arcs where all crossings occur between the these two arcs.  By planar isotopy, we can make one of these arcs straight, and we say the diagram is in {\em straight position}.

\begin{definition}\label{def:strnum}
Given a knot $K$, the {\em straight number} of $K$, $\mathtt{str}(K)$, is the minimum number of crossings over all diagrams of $K$ that are in straight position.
\end{definition}

If a knot $K$ has $\mathtt{str}(K) = c(K)$, where $c(K)$ is the crossing number, then we say $K$ is {\em perfectly straight}.  The  horizontal arc through the middle of the diagram in straight position is called the {\em straight strand} and every crossing occurs on this strand.

The braid group $B_n$ on $n$ strands has $n-1$ generators $\s_1, \s_2, \dots, \s_{n-1}$, where $\s_i$ represents $n$ vertical strands with the $i$-th strand passing over the $i+1$-st strand, traveling from top to bottom. An element $w$, of the braid group $B_n$, is turned into a link by taking the closure, $cl(w)$.  That link is a knot, or single component link, if the permutation obtained by mapping $w$ into the symmetric group is an $n$-cycle.

We also need to analyze behavior of knots under flypes, see \cite{flype} for more details.  For convenience, we include some definitions which will be used in future proofs.

\begin{definition}\label{def:flype}
A {\em flype} is move in a diagram described by Figure \ref{fig:flype}.  The single crossing that switches position to the other side of $F$ in a flype will be called the {\em flyper}.  \end{definition}

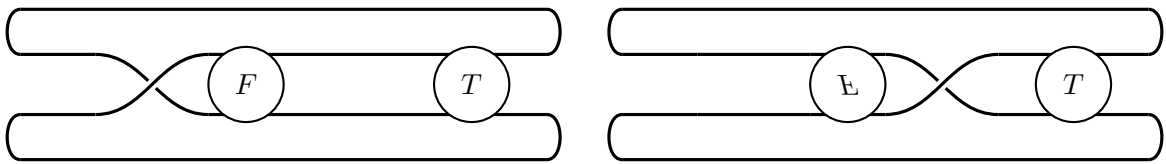
\begin{figure}[h]
\begin{tikzpicture}[scale=1]

\foreach \a in {.4}{
	\draw [very thick] (7,1+\a)   to (6,1+\a);
	\draw [very thick] (7,1-\a)   to (6,1-\a);
	
	\draw [very thick] (3,1+\a)   to (6,1+\a);
	\draw [very thick] (3,1-\a)   to (6,1-\a);
	
	\draw [very thick] (0,1+\a)   to (1,1+\a);
	\draw [very thick] (0,1-\a)   to (1,1-\a);
	
	\draw [very thick] (1,1+\a) [out=0,in=180]   to (2.5,1-\a);
	\draw [line width=0.14cm, white] (1,1-\a)  [out=0,in=180]  to (2.5,1+\a);
	\draw [very thick] (1,1-\a)  [out=0,in=180]  to (2.5,1+\a);
	
	\draw [very thick] (2.5,1+\a)   to (3,1+\a);
	\draw [very thick] (2.5,1-\a)   to (3,1-\a);

	\draw [very thick] (0,1+\a) [out=180,in=180]  to (0,2);
	\draw [very thick] (0,1-\a) [out=180,in=180]  to (0,0);
	
	\draw [very thick] (7,1+\a) [out=0,in=0]  to (7,2);
	\draw [very thick] (7,1-\a) [out=0,in=0]  to (7,0);

}

\draw [very thick] (0,2)   to (7,2);
\draw [very thick] (0,0)   to (7,0);

\draw [thick, fill = white] (3,1) circle [radius=.5];
\draw [thick, fill = white] (6,1) circle [radius=.5];

\node at (3,1) {$F$};
\node at (6,1) {$T$};

\begin{scope}[xshift = 8cm]

\foreach \a in {.4}{
	\draw [very thick] (7,1+\a)   to (5,1+\a);
	\draw [very thick] (7,1-\a)   to (5,1-\a);

	\draw [very thick] (3.5,1+\a) [out=0,in=180]   to (5,1-\a);
	\draw [line width=0.14cm, white] (3.5,1-\a)  [out=0,in=180]  to (5,1+\a);
	\draw [very thick] (3.5,1-\a)  [out=0,in=180]  to (5,1+\a);

	\draw [very thick] (0,1+\a)   to (1,1+\a);
	\draw [very thick] (0,1-\a)   to (1,1-\a);
	
	\draw [very thick] (1,1+\a) [out=0,in=180]   to (2.5,1+\a);
	\draw [very thick] (1,1-\a)  [out=0,in=180]  to (2.5,1-\a);
	
	\draw [very thick] (2.5,1+\a)   to (3.5,1+\a);
	\draw [very thick] (2.5,1-\a)   to (3.5,1-\a);

	\draw [very thick] (0,1+\a) [out=180,in=180]  to (0,2);
	\draw [very thick] (0,1-\a) [out=180,in=180]  to (0,0);
	
	\draw [very thick] (7,1+\a) [out=0,in=0]  to (7,2);
	\draw [very thick] (7,1-\a) [out=0,in=0]  to (7,0);

}

\draw [very thick] (0,2)   to (7,2);
\draw [very thick] (0,0)   to (7,0);

\draw [thick, fill = white] (3,1) circle [radius=.5];
\draw [thick, fill = white] (6,1) circle [radius=.5];

\node at (3,1) {{\scalebox{1}[-1]{$F$}}};
\node at (6,1) {$T$};

\end{scope}

\end{tikzpicture}

\caption{Two diagrams, $D_0$ on the left and $D_1$ on the right, related by a flype.}\label{fig:flype}
\end{figure}

\begin{definition}\label{def:cycle}\cite{Hoste}
Each crossing $x$ that is a flyper for some flype generates a unique {\em flype cycle} as shown in Figure \ref{fig:cycle}. This cycle is minimal in the sense that each tangle $F_i$ cannot be broken into to two nontrivial tangles that can each be part of a flype.  \end{definition}

\begin{figure}[h]
\begin{tikzpicture}[scale=1]

\foreach \a in {.4}{

	\draw [very thick] (0,1+\a) [out=0,in=180]   to (1.5,1-\a);
	\draw [line width=0.14cm, white] (0,1-\a)  [out=0,in=180]  to (1.5,1+\a);
	\draw [very thick] (0,1-\a)  [out=0,in=180]  to (1.5,1+\a);
	
	\draw [very thick] (1.5,1+\a)   to (7,1+\a);
	\draw [very thick] (1.5,1-\a)   to (7,1-\a);

	\draw [very thick] (9,1+\a)   to (11,1+\a);
	\draw [very thick] (9,1-\a)   to (11,1-\a);

	\draw [very thick] (0,1+\a) [out=180,in=180]  to (0,2);
	\draw [very thick] (0,1-\a) [out=180,in=180]  to (0,0);
	
	\draw [very thick] (11,1+\a) [out=0,in=0]  to (11,2);
	\draw [very thick] (11,1-\a) [out=0,in=0]  to (11,0);

}

\draw [very thick] (0,2)   to (11,2);
\draw [very thick] (0,0)   to (11,0);

\draw [thick, fill = white] (2,1) circle [radius=.5];
\node at (2,1) {$F_1$};

\draw [thick, fill = white] (4,1) circle [radius=.5];
\node at (4,1) {$F_2$};

\draw [thick, fill = white] (6,1) circle [radius=.5];
\node at (6,1) {$F_3$};

\draw [thick, fill = white] (10,1) circle [radius=.5];
\node at (10,1) {$F_n$};

\draw [fill] (8,1) circle [radius=.05];
\draw [fill] (8.5,1) circle [radius=.05];
\draw [fill] (7.5,1) circle [radius=.05];

\end{tikzpicture}

\caption{The flype cycle of length $n$ of a crossing that is a flyper.}\label{fig:cycle}
\end{figure}
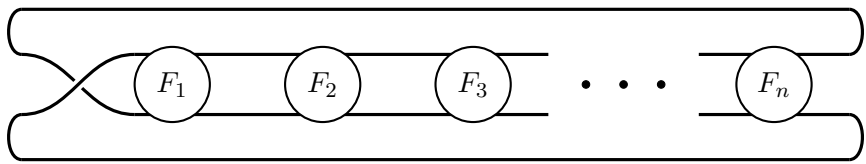



\section{Weaving knots}\label{sec:braids}

In \cite{spiral}, Brothers et al. define Spiral knots.  In this section, we prove alternating spiral knots, or weaving knots as in \cite{weave}, are not perfectly straight.  The idea for spiral knots comes from taking the standard diagram of a torus knot and changing the crossing information. 

\begin{definition}\label{def:spiral}\cite{spiral}
Let $n\geq 2$ and $m\geq1$, and let $K$ be the closure of the braid word $( \s_1^{\ep_{1}}\s_2^{\ep_{2}}\cdots \s_{n-1}^{\ep_{n-1}})^m$, where each $\ep_i = \pm 1$.  If $\ep = (\ep_1, \ep_2, \ldots, \ep_n)$, call $K$ the {\em spiral link} $S(n,m,\ep)$.
\end{definition}

Let $w = ( \s_1^{\ep_{1}}\s_2^{\ep_{2}}\cdots \s_{n-1}^{\ep_{n-1}})^m$ be the braid word of an spiral link.  Note that if the $\ep = (\pm1,\mp1,\pm1,\mp1, \ldots, \pm 1)$, then the spiral link is alternating and is called a {\em weaving link}, \cite{weave}, and call it $W(n,m)$.  See Figure \ref{fig:altbraid}.  For convenience, we name $w_0 = ( \s_1^{\ep_{1}}\s_2^{\ep_{2}}\cdots \s_{n-1}^{\ep_{n-1}})$ and thus, $w=w_0^m$.  From \cite{spiral}, when $\gcd(n,m)=1$, the spiral link is a knot.

\begin{figure}[h]
\begin{tikzpicture}[yscale=.35]

\foreach \j in {-2}{
\draw [very thick] (2,\j) [out=270,in=90]  to (1.5,\j -1);

\draw [line width=0.13cm,white] (1.5,\j) [out=270,in=90] to (2,\j-1);

\draw [very thick] (1.5,\j) [out=270,in=90] to (2,\j - 1);

\foreach \a in {2.5,3,3.5}
	\draw [very thick] (\a,\j) to (\a,\j-1);
}

\foreach \j in {-3}{
\draw [very thick] (2,\j) [out=270,in=90] to (2.5,\j - 1);

\draw [line width=0.13cm,white] (2.5,\j) [out=270,in=90] to (2,\j-1);
\draw [very thick] (2.5,\j) [out=270,in=90]  to (2,\j -1);

\foreach \a in {1.5,3,3.5}
	\draw [very thick] (\a,\j) to (\a,\j-1);
}


\foreach \j in {-4}{
\draw [very thick] (3,\j) [out=270,in=90]  to (2.5,\j -1);

\draw [line width=0.13cm,white] (2.5,\j) [out=270,in=90] to (3,\j-1);

\draw [very thick] (2.5,\j) [out=270,in=90] to (3,\j - 1);

\foreach \a in {1.5,2,3.5}
	\draw [very thick] (\a,\j) to (\a,\j-1);
}

\foreach \j in {-5}{
\draw [very thick] (3,\j) [out=270,in=90] to (3.5,\j - 1);

\draw [line width=0.13cm,white] (3.5,\j) [out=270,in=90] to (3,\j-1);
\draw [very thick] (3.5,\j) [out=270,in=90]  to (3,\j -1);

\foreach \a in {1.5,2,2.5}
	\draw [very thick] (\a,\j) to (\a,\j-1);
}

\end{tikzpicture}

\caption{An example of an alternating braid word $ w_0 =  \s_1\s_2^{-1}\s_3\s_4^{-1}\in B_5$.}\label{fig:altbraid}
\end{figure}
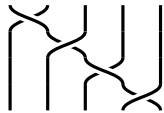

\begin{lemma}\label{lem:flype}
Let $n\geq3$ and $m\geq n+1$.  An weaving link $W(n,m)$ has exactly one reduced alternating diagram up to planar isotopy.
\end{lemma}

We omit the technical proof here and include a quick sketch.  Menasco and Thistlethwaite proved the Tait flyping conjecture \cite{flype} where, any two prime, oriented, reduced, alternating diagrams are connected by a sequences of flypes and planar isotopy.   By Menasco's result ``an alternating knot is prime if and only if it looks prime," \cite{prime}, we know that these $W(n,m)$ are all prime for $n\geq3$ and $m\geq n+1$.    Finally, by analyzing the circles in the diagram which intersect the knot 4 times, we can see that there only such possibilities surround a single crossing, and these cannot be flyped to obtain a new diagram. See Figure \ref{fig:circles} for the different types of candidates circles.  Hence, there is no flype that yields a different diagram for a knot $W(n,m)$.

\begin{figure}[h]
\begin{tikzpicture}[yscale=.5, xscale=.65]

\node at (-1,-1.5) {$w_1$};
\node at (-1,-4.5) {$w_2$};
\node at (-1,-7.5) {$w_3$};

\foreach \b in {-1,-2}{
	\draw[gray,  dashed] (.5,3*\b)   to (7.5,3*\b);
}

\foreach \b in {0,-1,-2}{

\draw [very thick,gray] (6.5,3*\b) [out=270,in=90]  to (6,3*\b-1.8) [out=270,in=90]  to (6.5,3*\b-3);
\draw [very thick] (1,3*\b) [out=270,in=90]  to (7,3*\b-3);
\foreach \a in {2,3,4,5, 6,7}
	\draw [very thick] (\a,3*\b) [out=270,in=90] to (\a-1,3*\b-3);
}

\draw [very thick, gray ] (3.2,-1.5) [out=270,in=220]  to (4.7,-1.8) [out=40,in=0]  to (4,-.9) [out=180,in=90]  to (3.2,-1.5);

\draw [very thick, gray ]  (.8,-6)  [out=-90,in=220] to (2.0,-7.5)  [out=50,in=270] to (1.5,-6)   [out=90,in=-40] to (1.8,-3.8)   [out=140,in=40] to (1.5,-3.8)   [out=220,in=90] to (.8,-6);

\draw[ very thick, gray] (4.6,-4.2)  [out=0,in=90] to (4.1,-7)  [out=-90,in=0] to  (3.3,-8) [out=180,in=-90] to (3.5,-6) [out=90,in=180] to (4.6,-4.2) ;

\end{tikzpicture}

\caption{All of the circles, in gray, intersect the knot more than 4 times.}\label{fig:circles}
\end{figure}
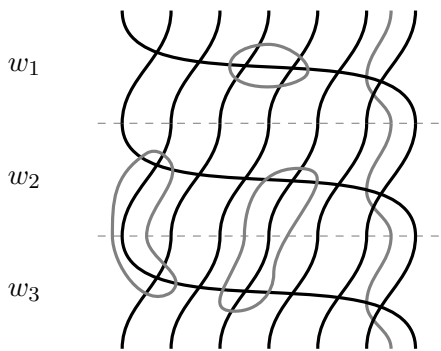

Now we prove the main theorem of this section.

\begin{theorem}\label{thm:spiral}~
Let $n\geq 3, m\geq n+1$ and $\gcd(n,m) = 1$.   Every weaving knot $W(n,m)$, is not perfectly straight, i.e. $\mathtt{str}(W(n,m))>c(W(n,m)).$
\end{theorem}

\begin{proof}
We will show that there is no way to traverse the knot and visit every crossing before visiting a crossing a second time.  By Lemma \ref{lem:flype} we know that this diagram of $K$ is the only diagram with the minimum number of crossing.  Thus we have $m(n-1)$ crossings in the diagram of $K$. 

Let $m=bn+r$, where $b\geq1$ and $1\leq r \leq n-1$.  Recall $w=w_0^m$ and label the $w_0$ in order $w_1$ through $w_m$.   We count the number of crossings we can reach before we come to a crossing the second time.  To maximize the number of crossings we meet in a single path traveling once down through the braid, we begin at strand one.  Notice that anytime we travel through a  single $w_i$ from strand one, we meet all $n-1$ crossings of that $w_i$.  Any other strand will only meet one crossing.  So when we start at strand one in $w_1$ and travel through $w_n$, we will have met $2(n-1)$ crossings.  So traveling all the way through  to $w_{bn}$ means we meet $b\cdot 2(n-1)$ crossings. Next, we will be at strand one in $w_{bn+1}$, which gives us another $n-1$ crossings.  The remaining of $w_i$'s to $w_m$ will have a single crossing each, giving us $r-1$ more crossings.  This means we have  $2b(n-1) + (n-1) + (r-1)$ total crossings by traveling once down through $w$, starting at strand one.  

We cannot obtain more crossings by following this path further because we will come back to the top of $w$ through the closure, and we have already meet all the crossings in $w_1$.  But we can travel backwards from strand one in $w_1$, where we began.  This will allow us to meet $r-1$ more crossings before we come to $w_{bn+1}$, where we were last in strand one and met every crossing.  Thus, we have met the maximum number of crossings, $$2b(n-1) + (n-1) + 2(r-1) =2m+n - 2(b+1).$$  By taking the symmetry of the braid into account, there is no other starting position which will increase the number of crossings we can meet.

Now, we claim that the maximum number of crossings in the diagram we can meet, $2m+n - 2(b+1)$, is less than the total number of crossings, $m(n-1)$, for any $n\geq 3$ and $m\geq n+1$.  When $n\geq 4$, we see that 
 $$2m+n - 2(b+1) < 2m+n < 3m \leq m(n-1),$$
  leaving us with when $n=3$.  So, assuming that $n=3$, we have $2m+n - 2(b+1) = 2m + 3 - 2(b+1)$, and since $b\geq 1$, we notice that 
$$2m + 3 - 2(b+1) < 2m = m(3-1),$$
 completing the proof.\end{proof}

This shows that the weaving knots are not perfectly straight, but does not give many clues as to what the straight number actually is.  Thus we ask the following question.

\begin{question}
What is the straight number of weaving knots?
\end{question}

We observe the following example which might shed some light on this question.

\begin{example}
The simplest example of a knot that this theorem applies to is $8_{18}$.  Let $n=3$, $m=4$, $w_0 = \s_1\s_2^{-1}$, and $w=w_0^4$ and take $K$ to be the closure $w$.  By Theorem \ref{thm:spiral} we see that this knot is not perfectly straight, so $\mathtt{str}(8_{18})>8.$  But from \cite{me}, by exhaustively checking all possible configurations straight knots of the standard table, we know that $\mathtt{str}(8_{18})=10$.  We can modify this problem slightly by letting $m=5$ and then $w=w_0^5$ to obtain the knot $10_{123}$, which we found to have  $\mathtt{str}(10_{123})=12$.  Thus, for both of these examples, we have  $\mathtt{str}(K)= c(K)+2$.
\end{example}

Therefore the bound  $\mathtt{str}(K)\geq c(K)+1$ that Theorem \ref{thm:spiral} gives us is not sharp for the smallest examples.  It seems unlikely that, in general, this theorem is sharp.  

\section{Adding full twists to diagrams}\label{sec:twists}

In this section, we describe a general method for generating families of knots that are not perfectly straight by adding twists into a diagram we know is not perfectly straight. 

\begin{definition}
A {\em twist region} of a diagram of a link K consists of maximal collections of bigon regions arranged end to end. A single crossing adjacent to no bigons is also a twist region. 
\end{definition}

For convenience, we introduce the following definition.

\begin{definition}\label{def:incfull}
We say we are modifying a diagram by {\em increasing the number of full twists} of a twist region when we add an even number of alternating crossings  to the twist region in such a way that no new crossings can be removed by Reidemeister  Type 2 moves.  
\end{definition}

For the main theorem of this section, we need the following useful lemma that we will then build upon.

\begin{lemma}\label{lem:incDia}
Given any diagram $D$, let $D'$ be the knot diagram obtained by increasing the number of full twists in any twist region of $D$.  Then $D$ in straight position if and only if $D'$ is in straight position.
\end{lemma}

\begin{proof}
Assume that $D$ is in straight position.  Then there is an arc $A$ which meets every crossing without meeting itself.  Any twist region $t$ is then in $A$.  Increasing the number of full twists of $t$ will make a new diagram $D'$ with $A$ now containing the new crossings.  For the other direction, assume that we increased the number of full twists of a twist region $t$ in a diagram $D$ to obtain the new diagram $D'$.  If $D'$ is in straight position, then the same arc $A$ without the new twists will also meet every crossing in $D$.\end{proof}

Next, we prove the two pieces that make up Theorem \ref{thm:inctwist}.

\begin{proposition}\label{prop:altstr}
Let $K$ be an alternating knot that is perfectly straight.  Given any minimal diagram $D$ of $K$, let $K'$ be the knot obtained by increasing the number of full twists in any twist region of $D$.  Then $K'$ is also perfectly straight.  
\end{proposition}

\begin{proof}
Given an alternating, perfectly straight knot $K$ and any minimal diagram $D$, we see that $D$ is a reduced alternating diagram.  Either $D$ is in straight position or not.  If it is in straight position, then increasing the number of full twists does not change this by Lemma \ref{lem:incDia}.  If $D$ is not in straight position, there is some sequence $S=(f_1,f_2,\ldots, f_n)$ of flypes which will create a diagram $D_0$ which is in straight position for the knot $K$.  Note that if we need to move a single crossing $x$ through $S$, we need to move every crossing of the twist region, $t$, that it belongs to.  Also, let $S'\subset S$ be the subsequence of flypes for which the crossings of $t$ are the flypers.   

Assume there are two more crossings in $t$ and call this diagram $D'$ of $K'$.  Let $\hat{S} $ be the sequence which we produce in the following way.  For each $f_i$ in $S$, add $f_i$ to $\hat{S}$ once if $f_i \not \in S'$ and three times if $f_i \in S'$.  Then by applying the sequence $\hat{S}$ to $D'$, we will move the two new crossings into position next to the other crossings of $t$ and obtain a new diagram $D_0'$.  

We claim that this diagram $D_0'$ is in straight position.  If we were to remove two crossings from $t$, then we would have our diagram $D_0$ of $K$ which is in straight position.  Then by applying Lemma \ref{lem:incDia}, we know that $D_0'$ is in straight position.  \end{proof}

Note that there are alternating knots which are perfectly straight but have minimal diagrams which are not in straight position, hence the need for this proposition.  Thus, flyping can change whether a minimal diagram is in straight position.  The knot $7_7$ is an example of this behavior.  The diagram in Rolfsen's standard table is not in straight position but $\mathtt{str}(7_7) = 7$.

\begin{proposition}
Let $K$ be an alternating knot that is not perfectly straight.  Given any minimal diagram $D$ of $K$, let $K'$ be the knot obtained by increasing the number of full twists in any twist region of $D$.  Then $K'$ is also not perfectly straight.  
\end{proposition}

\begin{proof}
Let $t$ be the twist region and $x$ a crossing on one end of $t$.  Then, by Definition \ref{def:flype}, $x$ is either a flyper of some flype or not -- we may assume the flype is not the rest of $t$.  First, assume $x$ is not a flyper for any flype, and we increase the number of full twists in $t$.  Then for each diagram $D'$ of $K'$, there is a corresponding diagram $D$ for $K$ that only differs by exactly the new full twists, which we know is not in straight position.  Then by applying Lemma \ref{lem:incDia} to each of these diagrams we know that $K'$ is still not perfectly straight.  

Next we assume that $x$ is a flyper for some flype and thus has a unique flype cycle, see Definition \ref{def:cycle}.  We may assume that $t$ is the single crossing $x$, as we can make the rest of $t$ a tangle new $F_i$.  There are three cases for what can happen when we increase the number of full twists of $t=x$ by one and then consider all flypes: 
\begin{enumerate}
\item all three crossings are in the same position between $F_i$ and $F_{i+1}$,
\item two crossings are in one position and the third crossing is in another position, or,
\item all three crossings are in different positions.
\end{enumerate}

If we are in case (1), then by Lemma \ref{lem:incDia} we know that this diagram is not in straight position as this is just increasing the number of full twists and making no flypes.  

For case (2), we have flyped a single crossing to another position.  This is equivalent to adding in two crossings to a position where $x$ is not.  Assume to the contrary that this diagram is in straight position.  Then the arc which meets every crossing would have been able to also meet every crossing in $D$, the diagram with these two crossings removed.  

Finally, for case (3), we have a similar situation as case (2).  Relabel the diagram so that we have one crossing to the left of the tangle $F_1$, just as in Figure \ref{fig:cycle}.  Then the other two crossings are at positions before tangle $F_i$ and after some $F_{i+j}$, flipping all the tangles from $i$ to $i+j$ upside-down.  Again, if we assume to the contrary that this diagram is in straight position, the diagram without these two extra crossings and all the tangles right-side up would be in straight position, contradicting our assumption and finishing the proof. \end{proof}

Combining these two propositions give us the following theorem. 

\begin{theorem}\label{thm:inctwist}
Let $K$ be an alternating knot. Given any minimal diagram $D$ of $K$, let $K'$ be the knot obtained by increasing the number of full twists in any twist region of $D$.  Then $K$ is perfectly straight if and only if $K'$ is perfectly straight.
\end{theorem}

From this theorem, one can quickly produce entire families of knots which are perfectly straight or not perfectly straight.  Here, we generalize weaving knots. 

\begin{corollary}\label{cor:spiral}
Let $n\geq 3, m\geq n+1$ and $\gcd(n,m)=1$.  Let  $w_i = ( \s_1^{\ep_{i_1}}\s_2^{\ep_{i_2}}\cdots \s_{n-1}^{\ep_{i_{n-1}}})$ where each $\ep_{i_j}$ is an odd integers.  Then let $w = w_1w_2\cdots w_m$ and let $K$ be the closure of $w$.  If $K$ is alternating, then $K$ is not perfectly straight.  
\end{corollary}

\begin{proof}
If each  $\ep_{i_j}=\pm1$ and $K$ is alternating, then $K$ is a weaving knot and thus, by Theorem \ref{thm:spiral}, $K$ is not perfectly straight. If  some $\ep_{i_j}$ are odd integers and $K$ is alternating, then by applying Theorem \ref{thm:inctwist}, we still have that $K$ is not perfectly straight.\end{proof}

\begin{question}
Is a similar statement to Theorem \ref{thm:inctwist} true when $K$ is not alternating?
\end{question}

Another application of this theorem is to take any alternating knot $K$ which has $\mathtt{str}(K) = c(K)+1$ and produce from it a family of knots which is has  $\mathtt{str}(K_n) = c(K_n)+1$.  For example, consider the knot $9_{32}$ in Figure \ref{fig:9_32}.  It is alternating and $\mathtt{str}(9_{32})=10$, which we showed in \cite{me}. To go from the reduced diagram on the left to the diagram in straight position, we need to deal with the single crossing not on the straight strand.  Notice that if we push the gray, under arc up under the straight strand and the bigon we introduce two crossings on the straight strand but eliminate the single crossing not on the straight strand.  The diagram on the right of Figure \ref{fig:9_32} is now in straight position.

Theorem \ref{thm:inctwist} tells us that we can take any twist region in the alternating diagram on the left, increase the number of full twists, and get a knot which is also not perfectly straight.  But if we only use the twist regions identified in the template in Figure \ref{fig:template}, we will produce knots which are not perfectly straight, but still have the property that  $\mathtt{str}(K) = c(K)+1$.

\begin{figure}[h]
\begin{tikzpicture}[scale=.5]

\draw [ultra thick] (-1,0) arc [radius=1, start angle=180, end angle= 0];
\draw [ultra thick] (2,0) arc [radius=3.5, start angle=180, end angle= 0];
\draw [ultra thick] (3,0) arc [radius=2.5, start angle=180, end angle= 0];
\draw [ultra thick] (4,0) arc [radius=1.5, start angle=180, end angle= 0];
\draw [ultra thick] (5,0) arc [radius=.5, start angle=180, end angle= 0];

\draw [ultra thick] (0,0) arc [radius=3.5, start angle=180, end angle= 360];
\draw [ultra thick] (-1,0) arc [radius=4.5, start angle=180, end angle= 360];
\draw [ultra thick] (3,0) arc [radius=.5, start angle=180, end angle= 360];

\draw [ultra thick, gray] (2,0) arc [radius=2, start angle=180, end angle= 360];

\draw [line width = .17cm, white] (1,0) arc [radius=2, start angle=180, end angle= 360];
\draw [ultra thick] (1,0) arc [radius=2, start angle=180, end angle= 360];

\foreach \a in {.13}{
\draw [line width = .17cm, white] (2*\a,0) -- (2-2*\a,0);
\draw [line width = .17cm, white] (2+2*\a,0) -- (4-2*\a,0);
\draw [line width = .17cm, white](4+2*\a,0) -- (6-2*\a,0);
\draw [line width = .17cm, white] (6+2*\a,0) -- (8-2*\a,0);
\draw [line width = .17cm, white] (8+2*\a,0) -- (9-2*\a,0);

\draw [ultra thick] (0,0) -- (2-\a,0);
\draw [ultra thick] (2+\a,0) -- (4-\a,0);
\draw [ultra thick] (4+\a,0) -- (6-\a,0);
\draw [ultra thick] (6+\a,0) -- (8-\a,0);
\draw [ultra thick] (8+\a,0) -- (9,0);
}

\draw [fill] (9,0) circle [radius=0.04];
\draw [fill] (0,0) circle [radius=0.04];

\begin{scope}[xshift = 12cm]
\draw [ultra thick] (-1,0) arc [radius=1, start angle=180, end angle= 0];
\draw [ultra thick] (2,0) arc [radius=4.5, start angle=180, end angle= 0];
\draw [ultra thick] (3,0) arc [radius=3.5, start angle=180, end angle= 0];
\draw [ultra thick] (4,0) arc [radius=2.5, start angle=180, end angle= 0];
\draw [ultra thick, gray] (5,0) arc [radius=1.5, start angle=180, end angle= 0];
\draw [ultra thick] (6,0) arc [radius=.5, start angle=180, end angle= 0];

\draw [ultra thick] (0,0) arc [radius=4.5, start angle=180, end angle= 360];
\draw [ultra thick] (-1,0) arc [radius=5.5, start angle=180, end angle= 360];
\draw [ultra thick] (3,0) arc [radius=.5, start angle=180, end angle= 360];
\draw [ultra thick , gray] (7,0) arc [radius=.5, start angle=180, end angle= 360];

\draw [ultra thick, gray] (2,0) arc [radius=1.5, start angle=180, end angle= 360];

\draw [ultra thick] (1,0) arc [radius=2.5, start angle=180, end angle= 360];

\foreach \a in {.13}{
\draw [line width = .17cm, white] (2*\a,0) -- (2-2*\a,0);
\draw [line width = .17cm, white] (2+2*\a,0) -- (4-2*\a,0);
\draw [line width = .17cm, white](4+2*\a,0) -- (7-2*\a,0);
\draw [line width = .17cm, white] (7+2*\a,0) -- (10-2*\a,0);
\draw [line width = .17cm, white] (10+2*\a,0) -- (11-2*\a,0);

\draw [ultra thick] (0,0) -- (2-\a,0);
\draw [ultra thick] (2+\a,0) -- (4-\a,0);
\draw [ultra thick] (4+\a,0) -- (7-\a,0);
\draw [ultra thick] (7+\a,0) -- (10-\a,0);
\draw [ultra thick] (10+\a,0) -- (11,0);
}

\draw [fill] (11,0) circle [radius=0.04];
\draw [fill] (0,0) circle [radius=0.04];

\end{scope}

\end{tikzpicture}
\caption{The knot $9_{32}$. On the left is a reduced diagram that is not in straight position. On the right is diagram of $9_{32}$ in straight position with 10 crossings. }\label{fig:9_32}
\end{figure}
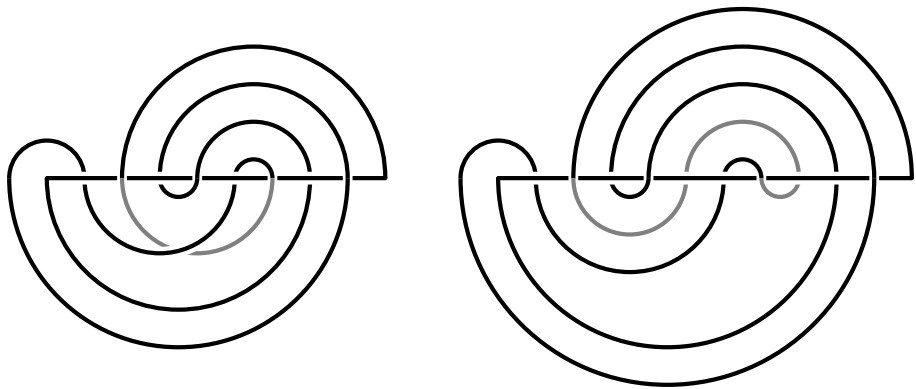

\begin{figure}[h]
\begin{tikzpicture}[scale=.42]

\draw [ultra thick] (-1,.5) arc [radius=1, start angle=180, end angle= 0];
\draw [ultra thick] (3,.5) arc [radius=5, start angle=180, end angle= 0];
\draw [ultra thick] (5,.5) arc [radius=3, start angle=180, end angle= 0];
\draw [ultra thick] (6,.5) arc [radius=1.5, start angle=180, end angle= 0];

\draw [ultra thick] (0,-.5) arc [radius=5, start angle=180, end angle= 360];
\draw [ultra thick] (-1,-.5) arc [radius=6.5, start angle=180, end angle= 360];

\draw [ultra thick, gray] (4,-.5) arc [radius=2, start angle=180, end angle= 360];

\draw [line width = .17cm, white] (2,-.5) arc [radius=2.5, start angle=180, end angle= 360];
\draw [ultra thick] (2,-.5) arc [radius=2.5, start angle=180, end angle= 360];

\draw [ultra thick] (0,0) -- (1,0);
\draw [ultra thick] (2,0) -- (3,0);
\draw [ultra thick] (4,0) -- (5,0);
\draw [ultra thick] (6,0) -- (7,0);
\draw [ultra thick] (8,0) -- (9,0);
\draw [ultra thick] (10,0) -- (11,0);
\draw [ultra thick] (12,0) -- (13,0);

\draw [ultra thick] (-1,.5) -- (-1,-.5);
\draw [ultra thick] (0,0) -- (0,-.5);
\draw [ultra thick] (13,.5) -- (13,0);

\draw [fill=white, thick] (1,-.5) rectangle (2,.5);
\draw [fill=white, thick] (3,-.5) rectangle (4,.5);
\draw [fill=white, thick] (5,-.5) rectangle (6,.5);
\draw [fill=white, thick] (7,-.5) rectangle (8,.5);
\draw [fill=white, thick] (9,-.5) rectangle (10,.5);
\draw [fill=white, thick] (11,-.5) rectangle (12,.5);

\node at (1.5,0) {$t_1$};
\node at (3.5,0) {$t_2$};
\node at (5.5,0) {$t_3$};
\node at (7.5,0) {$t_4$};
\node at (9.5,0) {$t_5$};
\node at (11.5,0) {$t_6$};

\draw [fill] (13,0) circle [radius=0.04];
\draw [fill] (0,0) circle [radius=0.04];

\end{tikzpicture}
\caption{The template for Theorem \ref{thm:template}.  It is made by increasing the number of full twists of a diagram of $9_{32}$. }\label{fig:template}
\end{figure}
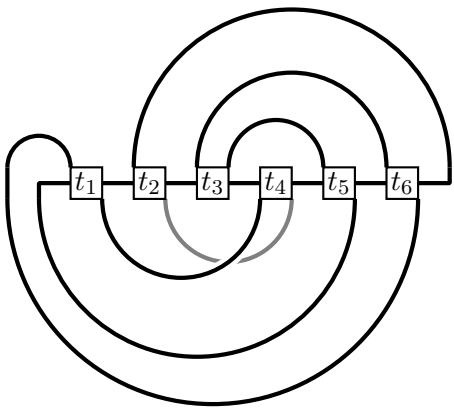

This template in Figure \ref{fig:template} has seven twist regions, one of which is a single crossing not on the straight strand, which we will not modify.  The twist regions $t_1,t_2,t_5,$ and $t_6$ need to have an odd number of crossings and $t_3$ and $t_4$ need to have an even number of crossings.  Then, make the knot alternating according to the one crossing defined in the template.

\begin{theorem}\label{thm:template}
Let $t=(t_1t_2,\ldots,t_6)$ be positive integers such that $t_1,t_2,t_5,$ and $t_6$  are odd and $t_3$ and $t_4$ are even and let $s$ be the sum of the $t_i$'s. Let $K_t$ be the alternating knot obtained from the template in Figure \ref{fig:template} with $t_i$ crossings in the corresponding twist region.  Then $\mathtt{str}(K_t) = c(K_t)+1 =  s+2$.
\end{theorem}

\begin{proof}
By Theorem \ref{thm:inctwist}, we know that $K_t$ is not perfectly straight for any choice of $t$.  We can put the diagram in straight position by pushing the gray arc up in the same way as in Figure \ref{fig:9_32}.  This gives us one more crossing, completing the proof.\end{proof}

It should be noted that $9_{32}$ is only special in the sense that it is the first knot to have enough crossings and $\mathtt{str}(K) = c(K)+1$.  This process can easily create an infinite number of infinite families, and we will briefly describe one such process later.

A reasonable question is what happens when we take the minimal diagram on the left in Figure \ref{fig:9_32} and increase the number of full twists of the one crossing {\em not} on the straight strand.  Inserting one full twist here turns $9_{32}$ into $11_{91}$.  By Theorem \ref{thm:template}, we know that increasing the number of full twists in any of the other twist regions preserves the fact that $\mathtt{str}(K) = c(K)+1$.  But by previous work done for \cite{me} that was not published, we know that  $\mathtt{str}(11_{91})=13 = c(11_{91}) +2$.  Hence, in general we cannot make our statement in Theorem \ref{thm:inctwist} stronger to include how straight number versus crossing number behaves.  It seems likely that this process of increasing the number of full twists only increases the difference between straight number and crossing number.  This leads us to the following conjecture.

\begin{conjecture}
Let $K$ be an alternating knot with $\mathtt{str}(K) = c(K)+n$, where  $n\geq1$, and let $K'$ be the knot obtained by increasing the number of full twists in some twist region.  Then $\mathtt{str}(K') \geq c(K')+n$.
\end{conjecture}

To create an infinite number of families like the one in Theorem \ref{thm:template}, we merely modify the diagram in careful way.  We sketch one possible way to do this here.  

\begin{example}
Consider the twist region $t_3$ in Figure \ref{fig:template}.    Break $t_3$ into two regions side by side and pull the strand that connects them up and follow the region of the diagram to the section of the straight strand between $t_5$ and $t_6$. Here introduce a new twist region $t_{x_1}$ with the straight strand.  Again, make the diagram alternating.  One can check the flypes, which will be the same as in the original diagram of $9_{32}$, to be sure this is not perfectly straight.  Using this diagram, make a similar template and you have a another infinite family.  

Use this trick again and break up the new twist region $t_{x_1}$ and pull it back between the two pieces of $t_3$ creating a new twist region $t_{x_2}$.  Again, we have a new family and repeating this process will yield a new infinite family each time.  
\end{example}

\begin{question}
Can we find lower bounds on straight number, from other other invariants,  which are strictly greater than crossing number?  

\end{question}



\end{document}